\documentclass[11pt,a4paper,reqno]{article}%
\usepackage[margin=1.2in]{geometry}

\usepackage{tikz-cd} 
\usepackage{graphicx} 
\usepackage{adjustbox} 
\usepackage{enumerate} 
\usepackage{pgfplots}
\pgfplotsset{compat=1.18}

\usepackage{amssymb} 
\usepackage{amsmath} 
\usepackage{mathtools} 
\usepackage{amsfonts} 
\usepackage{amsthm} 
\usepackage{mathrsfs} 

\usepackage{esint}
\usepackage{epigraph} 
\usepackage{calligra}
\usepackage{array}
\usepackage{color}

\usepackage{subfiles} 
\usepackage{appendix} 
\usepackage{pgfplots, tikz, stmaryrd}
\usepackage[english]{babel} 
\usepackage[colorlinks,linkcolor=red,anchorcolor=green,citecolor=blue]{hyperref} 
\hypersetup{linktocpage = true}

\usepackage{xcolor}
\hypersetup{colorlinks,linkcolor={blue!50!black},citecolor={blue!50!black},urlcolor={blue!80!black}}

\usepackage{fouriernc}
\usepackage{fancyhdr}
\usepackage{lastpage}
\usepackage{xstring}

\usepackage{verbatim}

\newenvironment{itenv*}
  {\phantomsection\par\medskip\noindent\itshape}
  {\par\medskip}

\usepackage{cleveref}

\renewcommand{\k}{\Bbbk} 

\DeclareMathOperator{\rk}{rk} 
\DeclareMathOperator{\Hom}{Hom} 

\DeclareMathOperator{\sheafhom}{\mathscr{H}\text{\kern -3pt {\calligra\large om}}\,} 

\DeclareMathOperator{\cA}{\mathcal{A}} 
\DeclareMathOperator{\cH}{\mathcal{H}} 

\DeclareMathOperator{\cD}{\mathcal{D}} 

\usetikzlibrary{decorations.markings}
\tikzset{->-/.style={decoration={ markings, mark=at position #1 with
{\arrow{>}}},postaction={decorate}}}
\tikzset{-<-/.style={decoration={ markings, mark=at position #1 with
{\arrow{<}}},postaction={decorate}}}

\def\Stab{\mathrm{Stab}}

\def\gl2{\widetilde{\mathrm{GL}}^+_{2}(\mathbb R)}

\def\cA{\mathcal A}
\def\cB{\mathcal B}
\def\cC{\mathcal C}
\def\cD{\mathcal D}

\def\cH{\mathcal H}

\def\cP{\mathcal P}

\def\cT{\mathcal T}

\def\cW{\mathcal W}

\def\k{\mathbf{k}}

\theoremstyle{plain}
\newtheorem{theorem}{Theorem}[section]
\newtheorem{corollary}[theorem]{Corollary}

\theoremstyle{definition}
\newtheorem{definition}[theorem]{Definition}

\newtheorem{remark}[theorem]{Remark}

\newtheorem{lemma}[theorem]{Lemma}
\numberwithin{equation}{section}

\numberwithin{equation}{section}

\newtheorem{question}[theorem]{Question}

\usepackage{url}

\usetikzlibrary{matrix,positioning,decorations.markings,arrows,decorations.pathmorphing, 
backgrounds,fit,positioning,shapes.symbols,chains,shadings,fadings,calc}
\tikzset{->-/.style={decoration={ markings, mark=at position #1 with
{\arrow{>}}},postaction={decorate}}}
\tikzset{-<-/.style={decoration={ markings, mark=at position #1 with
{\arrow{<}}},postaction={decorate}}}

\newcommand{\keywords}[1]
{
  \small	
  \textbf{\textit{Keywords---}} #1
}

\title{A counterexample to the Jordan--H\"older property for polarizable semiorthogonal decompositions}
\author{Fabian Haiden and Dongjian Wu}

\newcommand{\Addresses}{{
  \setlength{\parindent}{0pt}
  \bigskip
  \footnotesize
    FH: \textsc{Centre for Quantum Mathematics, Department of Mathematics and Computer Science, University of Southern Denmark, Campusvej 55, 5230 Odense, Denmark} \par\nopagebreak
	\textit{E-mail:} \texttt{fab@sdu.dk}
    \medskip

    DW: \textsc{Department of Mathematical Sciences, Tsinghua University, 100084 Beijing, China} \par\nopagebreak
	\textit{E-mail:} \texttt{wdj20@mails.tsinghua.edu.cn}
}}

\date{}

\begin{document}


\maketitle

\begin{abstract}
We show that the Jordan--H\"older property fails for polarizable semiorthogonal decompositions --- those where every factor admits a Bridgeland stability condition. Counterexamples exist among Fukaya categories of surfaces and bounded derived categories of smooth projective varieties. Furthermore, we give an example of a smooth and proper pre-triangulated dg category with positive rank Grothendieck group which does not admit a stability condition. 
\end{abstract}


\tableofcontents


\setlength\parindent{0pt}
\setlength{\parskip}{5pt}

\section{Introduction}

\subsection*{The failure of the Jordan--H\"older property for SODs}

In the study of triangulated categories, such as derived categories of coherent sheaves on projective varieties, a central notion is that of a \textit{semiorthogonal decomposition} (SOD) $\mathcal C=\langle \mathcal C_1,\ldots,\mathcal C_n\rangle$. 
By definition, each $\mathcal C_i\subset\mathcal C$ is a full triangulated subcategory, $\mathrm{Hom}(\mathcal C_i,\mathcal C_j)=0$ for $j<i$, and $\mathcal C$ is the extension closure of the $\mathcal C_i$'s.

SODs are reminiscent of composition series, and an initial optimistic guess may have been that an analog of the Jordan--H\"older theorem holds for them.
A first counterexample is attributed to Bondal~\cite{kuznetsov_5ecm,kuznetsov2013simplecounterexamplejordanholderproperty}, where the triangulated category is the bounded derived category of finite-dimensional representations of the following quiver with relations:
\begin{equation}
\label{Q}
Q=\left(
\begin{tikzcd}
\mathop{\bullet}\limits_{1} \arrow[r, "b_1"', bend right] \arrow[r, "a_1", bend left] & \mathop{\bullet}\limits_{2} \arrow[r, "a_2", bend left] \arrow[r, "b_2"', bend right] & \mathop{\bullet}\limits_{3}
\end{tikzcd}\Bigg{\vert}\quad a_2b_1=b_2a_1=0
\right).
\end{equation}
Since the quiver is acyclic, the three projective representations form a full exceptional collection of length 3 in $D^b(Q)$.
On the other hand, the object $P$ given by the representation
\begin{equation}
\label{Q'}
P:
\begin{tikzcd}
\mathbf k \arrow[r, "1", shift left=2] \arrow[r, "0"', shift right] & \mathbf k \arrow[r, "1", shift left=2] \arrow[r, "0"', shift right] & \mathbf k
\end{tikzcd}
\end{equation}
is exceptional, but there are no exceptional objects left- or right-orthogonal to $P$, i.e. $P$ is not part of an exceptional collection of length $> 1$. 

In Bondal's example, the triangulated category is not of the form $D^b(X)$ for $X$ some smooth projective variety.
It was a hope of Kuznetsov, with applications to birational geometry in mind, that the Jordan--H\"older property could at least be true for categories of this form~\cite{kuznetsov_5ecm}.
A first counterexample was found by B\"ohning--Graf von Bothmer--Sosna \cite{BOHNING2014479}, where $X$ is the classical Godeaux surface.
Another, simpler counterexample was pointed out shortly thereafter by Kuznetsov himself, based on realizing Bondal's example as an admissible subcategory of the derived category of a blowup of $\mathbb P^3$~\cite{kuznetsov2013simplecounterexamplejordanholderproperty}.

Among geometric derived categories, very few have had the Jordan--H\"older property confirmed, including (trivially) connected Calabi--Yau categories, derived categories of curves of positive genus \cite{OKAWA20112869}, $\cD^b(\mathbb P^1)$, quotient stacks $\cD^b(\mathbb P^1/\Gamma)$ \cite{kirillov2006mckaycorrespondenceequivariantsheaves}, and $\cD^b(\mathbb P^2)$~\cite{pirozhkov}.

\subsection*{Stability conditions and polarizability}

The notion of a stability conditions on a triangulated category was introduced by Bridgeland~\cite{MR2373143}, drawing inspiration from geometric invariant theory and string theory. By far not all triangulated categories admit a stability condition in this sense.
Conjecturally, all triangulated categories of the form $D^b(X)$, where $X$ is a smooth projective variety, admit a stability condition. This is a central open question in the field.

In a very intriguing paper~\cite{halpernleistner2024noncommutativeminimalmodelprogram}, Halpern-Leistner proposes a conjectural picture involving birational geometry, semiorthogonal decompositions, and stability conditions.
As a main principle, he suggests restricting to those SODs $\mathcal C=\langle \mathcal C_1,\ldots,\mathcal C_n\rangle$ where each $\mathcal C_i$ admits a stability condition.
Following~\cite{halpernleistner2024noncommutativeminimalmodelprogram}, we call such SODs \textit{polarizable}.
A hope is expressed in~\cite{halpernleistner2024noncommutativeminimalmodelprogram}, that the Jordan--H\"older property might hold for this restricted class of SODs.

If we consider the example coming from the Godeaux~\cite{BOHNING2014479}, then one of the SODs is constructed already in~\cite{zbMATH06245837} and of the form
\begin{equation}\label{godeauxSOD}
D^b(X)=\langle\mathcal A,E_1,\ldots,E_{11}\rangle
\end{equation}
where $E_i$ are exceptional objects and $K_0(\mathcal A)=\mathbb Z/5$.
In particular, $\mathcal A$ has no stability conditions and so this SOD is not polarizable.

\subsection*{A nice category without stability condition}

Does Bondal's example provide a counterexample to the Jordan--H\"older property for \textit{polarizable} SODs?
This reduces to the following question: Does the category $P^{\perp}\subset D^b(Q)$ admit a stability condition?
Our first main result gives a negative answer to this question. 

\begin{theorem}[\Cref{thm:Existense of stab}]
The triangulated category $P^{\perp}$, equipped with the standard grading, does not admit a stability condition.
\end{theorem}

This example was previously considered by Sung~\cite{Sung2022RemarksOT}, who proved the non-existence of stability conditions invariant under the Serre functor $S$, i.e. a $\sigma$ with $S\cdot\sigma=\sigma\cdot g$, for some $g\in\widetilde{\mathrm{GL}}^+(2,\mathbb R)$.

This is to our knowledge the first example of a smooth and proper pre-triangulated dg-category $\mathcal A$ with $\mathrm{rk}(K_0(\mathcal A))=2$ which does not admit a stability condition.
It would be interesting to find a general obstruction which explains our result.
Note that the ``phantom'' $\mathcal A$ as in~\eqref{godeauxSOD} provides an example of a smooth and proper pre-triangulated dg-category with $\mathrm{rk}(K_0(\mathcal A))=0$ which does not admit a stability condition, and $\mathcal C\coloneqq\langle \mathcal A,E_1\rangle\subset D^b(X)$ provides an example with $\mathrm{rk}(K_0(\mathcal C))=1$, since any stability condition on $\mathcal C$ would need to have a finite length heart generated by a single simple object, which contradicts $K_0(\mathcal C)=\mathbb Z\oplus\mathbb Z/5$. We thank Daniel Halpern--Leistner for pointing this out to us.

\begin{question}
Does $P^{\perp}$ admit a bounded $t$-structure? (See Section~\ref{sec:Bondal_quiver} for a description of $P^\perp$ in terms of a graded gentle algebra.)
Known deep obstructions~\cite{AGH,Neeman1,Neeman2} do not apply in this case.
\end{question}

After this paper first appeared as a preprint, Shengxuan Liu informed us that the question is answered positively in his upcoming joint work with Perry~\cite{Shengxuan}.

Our proof of \Cref{thm:Existense of stab} is based on a geometric interpretation of $D^b(Q)$ and $P^{\perp}$ as Fukaya categories of surfaces.
Indeed, $D^b(Q)$ is the Fukaya category of a genus 1 surface with one boundary component and two marked points on the boundary, and $P^{\perp}$ is the Fukaya category of a genus 1 surface with one boundary component and \textit{one} marked point on the boundary, in both cases with the standard grading (constant line field).
This correspondence is recounted in Section~\ref{sec:Bondal_quiver}.
As a consequence, we also obtain the following.

\begin{corollary}
Let $S$ be a compact oriented surface with boundary, $M\subset\partial S$ a set of marked points such that every component of $\partial S$ contains a marked point, and $\eta$ an arbitrary grading structure (line field) on $S$. Then the following are equivalent:
\begin{enumerate}
    \item The Fukaya category $\mathcal W(S,M,\eta)$ admits a stability condition.
    \item $S$ is not a genus one surface with a single boundary component, one marked point on the boundary, and $\eta$ the standard grading.
\end{enumerate}
\end{corollary}

If one restricts to stability conditions with finite length heart in the above statement, then this is already implicitly contained in the work of Jin--Schroll--Wang~\cite[Theorem 1.1]{JSZ23} in view of~\cite[Proposition 5.12]{chang2023recollementspartiallywrappedfukaya}. 
It only remains to exclude the possibility of a stability condition without finite length heart on the Fukaya category in question, which is our Theorem~\ref{thm:Existense of stab}.

\subsection*{A counterexample to Jordan--H\"older for polarizable SODs}

It turns out that if we slightly modify Bondal's example by changing the grading, then we do get counterexamples to the Jordan--H\"older property for polarizable SODs.

\begin{theorem}[\Cref{thm:Polarizable J-H}]
Let $\mathcal C=D^b(Q)$, where $Q$ is the quiver with relations as in \eqref{Q}, but with grading such that $|a_1|\neq|b_1|$ or $|a_2|\neq|b_2|$, or both. 
Equivalently, $\mathcal C$ is the Fukaya category of a genus 1 surface with one boundary component with two marked points on the boundary and with non-standard grading.
Then $\mathcal C$ does not satisfy the Jordan--H\"older property for polarizable SODs.
More precisely, there exists both an exceptional collection of length 3 and a polarizable SOD of the form $\langle P,P^{\perp}\rangle$ where $P^{\perp}$ does not contain any exceptional object.
\end{theorem}

One may generalize these examples by increasing the genus, in which case the condition on the grading becomes unnecessary.

Combining the above with a powerful theorem of Orlov~\cite[Theorem 1.10]{Orlov2015GeometricRO} yields the following.

\begin{corollary}
There exists a smooth projective variety $X$ such that $D^b(X)$ does not satisfy the Jordan--H\"older property for polarizable SODs.
\end{corollary}

Namely, one can find an $X$ such that $D^b(X)=\langle\mathcal C,E_1,\ldots,E_n\rangle$, where $\mathcal C$ is as in Theorem~\ref{thm:Polarizable J-H} and $E_i$ are exceptional.

\subsection{Contents}

In \Cref{sec:pre}, we review the concepts of the notions of Bridgeland stability conditions on triangulated categories, the polarizable Jordan--H\"{o}lder property, partially wrapped Fukaya categories, and graded gentle algebras. In \Cref{sec:Bondal_quiver}, we explain the geometric interpretation of the bounded derived category of the quiver \eqref{Q} and the object \eqref{Q'} in terms of Fukaya categories of surfaces. In \Cref{sec:main results}, we demonstrate that $P^{\perp}=\cW(S,\{p\},\eta)$ equipped with the standard grading $\eta$ does not admit a stability condition and confirm that $\cW(S,M,\eta)=\mathrm{Perf}(\k Q)$ with non-standard grading does not satisfy the Jordan--H\"older property for polarizable SODs.

\subsection*{Acknowledgements}
FH thanks Wen Chang, Daniel Halpern-Leistner, Sibylle Schroll, Benjamin Sung, and Alex Takeda for useful discussions.
DW thanks the University of Southern Denmark for providing perfect working conditions during his visit, during which this research was conducted.
DW is also thankful to Yu Qiu for his continuous support and patience throughout his research.
FH is supported by the VILLUM FONDEN, VILLUM Investigator grant 37814 and Sapere Aude grant 3120-00076B from the Independent Research Fund Denmark (DFF).
This paper is partly a result of the ERC-SyG project Recursive and Exact New Quantum Theory (ReNewQuantum) which received funding from the European Research Council (ERC) under the European Union's Horizon 2020 research and innovation programme under grant agreement No 810573.

\section{Preliminaries}
\label{sec:pre}
In this section, we review the notions of Bridgeland stability conditions on triangulated categories, the polarizable Jordan--H\"{o}lder property, the partially wrapped Fukaya categories and graded gentle algebras, referring to \cite{MR2373143, Bondal1995SemiorthogonalDF, Bondal2002DerivedCO,halpernleistner2024noncommutativeminimalmodelprogram,HKK17,Lekili2018DerivedEO,OPS18,APS23}. All triangulated categories we consider in this paper have finitely generated Grothendieck group, $K_0(\cD)$. 

\subsection{Bridgeland stability conditions}

A \emph{Bridgeland pre-stability condition} $\sigma=(Z, \mathcal P)$ on a triangulated category $\mathcal D$ is given by a group homomorphism $Z: K_0(\cD)\to \mathbb C$, termed the \emph{central charge}, and a collection of full additive subcategories $\mathcal P(\phi)\subset \mathcal D$ for each $\phi\in \mathbb R$, called the \emph{slicing}. This pair is subject to the following axioms:
\begin{enumerate}
\item[(a)] if $0\ne E\in \mathcal P(\phi)$, then $Z(E) \in \mathbb R_{>0}\cdot\mathrm{exp}(\pi\textbf i\phi)$,
\item[(b)] for all $\phi\in \mathbb R$, $\mathcal P(\phi +1) = \mathcal P(\phi)[1]$,
\item[(c)] if $\phi_1>\phi_2$ and $A_i\in \mathcal P(\phi_i)$, $i=1,2$, then ${\rm Hom}_{\mathcal D}(A_1, A_2) = 0$,
\item[(d)] any $0\neq E\in \mathcal D$ has a {Harder--Narasimhan filtration} (HN filtration).
\end{enumerate}

For any $0\neq E\in\cD$ one defines $\phi^+_{\sigma}(E)$ and $\phi^-_{\sigma}(E)$, as the maximal and minimal phases, respectively, of semistable factors appearing in the HN filtration of $E$. An object $E\in \mathcal P(\phi)$ for some $\phi\in \mathbb R$ is called semistable, and in such a case, $\phi = \phi^{\pm}_{\sigma}(E)$. Moreover, if $E$ is a simple object in $\mathcal P(\phi)$, it is said to be stable. We define $\mathcal P(I)$ for an interval $I$ in $\mathbb R$ as
 \[
 \mathcal P(I) = \left\{E\in \mathcal D\ |\ \phi^{\pm}_{\sigma}(E)\in I\right\}\cup\{0\}.
 \]
Consequently, for any $\phi\in \mathbb R$, both $\mathcal P([\phi, \infty))$ and $\mathcal P((\phi, \infty))$ are t-structures in $\mathcal D$ with hearts $\cP([\phi,\phi+1))$ and $\cP((\phi,\phi+1])$, respectively. According to \cite{MR2373143}, each category $\cP(\phi)\subset\cD$ is abelian and $\cP(I)\subset\cD$ is quasi-abelian for any interval $I\subset\mathbb R$ of length $<1$. 

We now restrict our attention to Bridgeland pre-stability conditions that fulfill the \emph{support property} (cf. \cite{kontsevich2008stabilitystructuresmotivicdonaldsonthomas}). 
For this, we fix a finite rank lattice $\Lambda$ with a surjective group homomorphism $v:K_0(\cD)\twoheadrightarrow\Lambda$. 
A Bridgeland pre-stability condition $\sigma=(Z,\mathcal P)$ satisfies the support property (with respect to $(\Lambda,v)$) if $Z$ factors through $v$, i.e. $Z:K_0(\cC)\xrightarrow[]{v}\Lambda\to\mathbb C$, and there exists a constant $C>0$ and a norm $\Vert\cdot\Vert$ on $\Lambda\otimes_\mathbb Z{\mathbb R}$ such that for all $\sigma$-semistable objects $E\in\cD$, the following holds:
\begin{equation}
\label{support_property}
\Vert v(E)\Vert\le C\vert Z(E)\vert.
\end{equation}
A Bridgeland pre-stability condition that satisifes the support property is referred to as a \emph{Bridgeland stability condition}. Unless there is a specific need to emphasis $(\Lambda,v)$, we will denote the set of stability conditions with respect to $(\Lambda,v)$ simply by $\Stab(\cD)$. According to \cite[Proposition 5.3]{MR2373143}, a stability condition can also be characterized by a pair $(Z, \cH)$, in which $\cH$ constitutes a heart of the category $\cD$, and $Z$ represents a stability function on $\cH$ that adheres to the Harder-Narasimhan condition. 


\subsection{Polarizable Jordan--H\"{o}lder property}

A \emph{semiorthogonal decomposition} of a triangulated category $\cD$, denoted by $\cD=\langle\cD_1,\dots,\cD_n\rangle$, is a totally ordered collection of full triangulated subcategories of $\cD_i\subset\cD$ such that 
\begin{enumerate}[$(1)$]
\item  $\Hom_{\cD}(\cD_i,\cD_j)=0$ for all $1\le j<i\le n$;
\item $\cD=\mathrm{thick}\langle\cD_1,\dots,\cD_n\rangle$, i.e. $\cD$ is the smallest triangulated subcategory containing $\cD_1$, \dots, $\cD_n$.
\end{enumerate}
 An object $E$ is \emph{exceptional} if $\Hom(E,E)=\k$ and $\Hom(E,E[t])=0$ for $t\ne0$. An \emph{exceptional collection} consists of exceptional objects ${E_1, E_2, \dots, E_m}$ such that $\Hom(E_i,E_j[t])=0$ for all $i>j$ and all $t\in\mathbb Z$. An exceptional collection in $\cD$ yields a semiorthogonal decomposition
\[
\cD=\langle\cA,E_1,\dots,E_m\rangle,
\]
where $E_i$ represents the subcategory generated by the same exceptional object, and 
\[
\cA:=\langle E_1,\dots,E_m\rangle^{\perp}=\left\{A\in\cD\colon\Hom(E_i,A[p]=0),\text{ for all } p,i\right\}. 
\]
Similarly, it also induces the semiorthogonal decomposition
\[
\cD=\langle E_1,\dots,E_m,\cA'\rangle
\]
where \[
\cA'={^{\perp}\langle E_1,\dots,E_m\rangle}=\left\{A'\in\cD\colon\Hom(A',E_i[p])=0,\text{ for all } p,i\right\}.
\]If $\cA=0$, then the exceptional collection ${E_1,\dots,E_m}$ is referred to as \emph{full}.

A triangulated category $\cT$ has the \textit{Jordan--H\"{o}lder property} if for any pair 
\[
\cT=\langle\cA_1,\dots,\cA_n\rangle=\langle\cB_1,\dots,\cB_m\rangle
\]
of \emph{maximal} semiorthogonal decompositions, i.e. each factor does not have any nontrivial semiorthogonal decomposition, one has $m=n$ and there is a permutation $\sigma\in S_m$ such that $\cB_i\cong\cA_{\sigma(i)}$ for each $1\le i\le m$. \cite{kuznetsov2013simplecounterexamplejordanholderproperty} and \cite{BOHNING2014479} present examples in the derived categories of projective varieties that demonstrate the failure of the Jordan--Hölder property. We now turn our attention to the polarizable Jordan--Hölder property:

\begin{definition}
\label{def:Polarizable J-H}
A triangulated category $\cT$ is said to possess the polarizable Jordan--Hölder property if it fulfills the Jordan--Hölder property for all \emph{polarizable maximal semi-orthogonal decompositions}, meaning that each factor in the maximal semiorthogomal decomposition admits a stability condition.
\end{definition}


\subsection{Partially wrapped Fukaya categories and graded gentle algebras}
\label{sec:PWF&GGA}

The partially wrapped Fukaya category $\cW(S,M,\eta)$, with coefficients in a field $\k$, is associated with a graded surface $(S,M,\eta)$, where $S$ denotes an oriented, connected, compact surface with a non-empty boundary $\partial S$. The set $M$ is a collection of marked points on $\partial S$, referred to as stops. Additionally, $\eta$ represents a line field defined on $S$. Given a graded surface $(S,M,\eta)$, an \emph{arc} of $S$ is defined as an embedded curve $\alpha$ such that $\partial\alpha\subset M$, $\alpha$ is transverse to $M$ and $\partial\alpha$ is not isotopic to a subset of $M$. A \emph{boundary arc} is an arc isotopic to a boundary component not lying within $M$. An \emph{arc system} is defined as a collection $A=\{\alpha_i\}$ of arcs that are pairwise disjoint and non-isotopic. This system is called a \emph{full formal arc system} if it results in a polygonal decomposition of the surface where each polygon contains exactly a single boundary arc that is not belonging to $A$. Each arc $\alpha$ is regarded as a \emph{graded arc} by fixing a homotopy $\widetilde{\alpha}$ between the restriction of $\eta$ to $\alpha$ and the line field  $\dot{\alpha}$ along $\alpha$ determined by the tangent lines of $\alpha$. For two graded arcs $\alpha$ and $\beta$ that meet a common endpoint $q\in M$, an \emph{oriented angle} $a$ from $\alpha$ to $\beta$ at $q$ is an oriented, embedded path $u_a$ in a neighborhood of $q$, connecting $\alpha$ to $\beta$ with $q$ to the right, and ensuring that $u_a$ does not intersect any other arc in $A$.
The \emph{degree} of the oriented angle $a$ is determined by the grading of the concatenation of $\alpha$ and $\beta$. We refer to \cite{HKK17} for a combinatorial description of the partially wrapped Fukaya category. 

Recall that a \emph{graded gentle algebra} is a graded algebra of the form $A=\k Q/I$ where $Q=(Q_0,Q_1)$ is a finite graded quiver (i.e. to each arrow is assigned a degree in $\mathbb Z$) and $I$ an ideal of the path algebra $\k Q$ generated by paths of length two such that:
\begin{enumerate}
    \item Each vertex in $Q_0$ is the source of at most two arrows and the target of at most two arrows.
    \item For each arrow $\alpha\in Q_1$, there is at most one arrow $\beta$ such that $0\ne\alpha\beta\in I$ and at most one arrow $\gamma$ such that $0\ne\gamma\alpha\in I$. Furthermore, there is at most one arrow $\beta'$ such that $\alpha\beta'\notin I$ and at most one arrow $\gamma'$ such that $\gamma'\alpha\notin I$.
\end{enumerate}
For a full formal arc system $A=\{\alpha_i\}$ on a graded surface $(S,M,\eta)$, one can construct a graded gentle algebra $\k Q(A)/I(A)$ where:
\begin{enumerate}
    \item The vertices of the quiver $Q(A)$ correspond one-to-one with the arcs in $A$;
    \item Each oriented angle between arcs $\alpha_i$ to $\alpha_j$ induces an arrow $a_{ij}$ from the vertices $\alpha_i$ and $\alpha_j$ with the degree being the same as the degree of the oriented angle;
    \item The ideal $I(A)$ is generated by the length two paths of the form $ab$ for all oriented angles with $u_a:\alpha\to\beta$ and $u_b:\beta\to\gamma$.
\end{enumerate}

According to \cite{Lekili2018DerivedEO,OPS18}, every graded gentle algebra can be realized by some graded surface. We now review the following result:
\begin{theorem}[\cite{HKK17,Lekili2018DerivedEO}]
\label{equi:PWF-GGA}
Let $(S,M,\eta)$ be the graded surface associated to a graded gentle algebra $A$. If $A$ is homologically smooth and proper, then we have the triangle equivalence 
\[
\cW(S,M,\eta)\cong \mathrm{Perf}(A).
\]
\end{theorem}


\section{Genus one examples}
\label{sec:Bondal_quiver}

In this section, we explain the (well known to experts) geometric interpretation of the bounded derived category of the quiver \eqref{Q} and the object \eqref{Q'} in terms of Fukaya categories of surfaces.
Fix throughout a coefficient field $\k$.

Let $S$ be an oriented surface of genus one with one boundary component, $M=\{p,q\}\subset\partial S$ a pair of marked points on the boundary of $S$, and $\eta$ and arbitrary grading on $S$.
The Fukaya category $\cW(S,M,\eta)$ is a smooth and proper $A_\infty$-category over $\k$.

We choose a full formal arc system $A$ on $(S,M)$ with three arcs as indicated on the right in Figure~\ref{figure:marked surfaces}.
The corresponding graded gentle algebra has underlying ungraded algebra described by the quiver with relations $Q$ as in \eqref{Q}, where the three vertices correspond to the three arcs, $a_1$ and $a_2$ come from the intersections of the arcs at $p\in M$ and $b_1$, and $b_2$ come from the intersections of the arcs at $q\in M$.
The $\mathbb{Z}$-grading on the arrows is determined by $\eta$.
The three arcs correspond to three objects, $P_1,P_2,P_3$, of $\cW(S,M,\eta)=\mathrm{Perf}(\k Q)$.

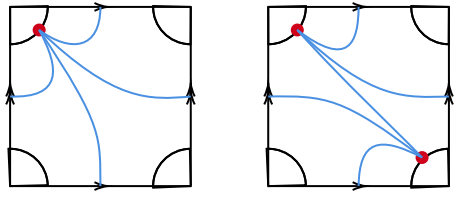
\begin{figure}
\tikzset{every picture/.style={line width=0.75pt}}
\begin{tikzpicture}[x=0.75pt,y=0.75pt,yscale=-1,xscale=1]
\draw  [draw opacity=0] (306.03,166.9) .. controls (306.02,166.76) and (306.02,166.62) .. (306.01,166.48) .. controls (305.81,156.14) and (314.01,147.58) .. (324.34,147.34) -- (324.79,166.11) -- cycle ; \draw   (306.03,166.9) .. controls (306.02,166.76) and (306.02,166.62) .. (306.01,166.48) .. controls (305.81,156.14) and (314.01,147.58) .. (324.34,147.34) ;  
\draw    (105.85,75.97) -- (105.85,166.11) ;
\draw [shift={(105.85,114.44)}, rotate = 90] [color={rgb, 255:red, 0; green, 0; blue, 0 }  ][line width=0.75]    (10.58,-1.97) .. controls (8.19,-0.84) and (6.01,-0.18) .. (4.03,0) .. controls (6.01,0.18) and (8.19,0.84) .. (10.58,1.97)(6.56,-1.97) .. controls (4.17,-0.84) and (1.99,-0.18) .. (0,0) .. controls (1.99,0.18) and (4.17,0.84) .. (6.56,1.97)   ;
\draw [color={rgb, 255:red, 0; green, 0; blue, 0 }  ,draw opacity=1 ][line width=0.75]    (105.85,75.97) -- (195.99,75.97) ;
\draw [shift={(154.52,75.97)}, rotate = 180] [color={rgb, 255:red, 0; green, 0; blue, 0 }  ,draw opacity=1 ][line width=0.75]    (6.56,-1.97) .. controls (4.17,-0.84) and (1.99,-0.18) .. (0,0) .. controls (1.99,0.18) and (4.17,0.84) .. (6.56,1.97)   ;
\draw    (105.85,166.11) -- (195.99,166.11) ;
\draw [shift={(154.52,166.11)}, rotate = 180] [color={rgb, 255:red, 0; green, 0; blue, 0 }  ][line width=0.75]    (6.56,-1.97) .. controls (4.17,-0.84) and (1.99,-0.18) .. (0,0) .. controls (1.99,0.18) and (4.17,0.84) .. (6.56,1.97)   ;
\draw    (195.99,75.97) -- (195.99,166.11) ;
\draw [shift={(195.99,114.44)}, rotate = 90] [color={rgb, 255:red, 0; green, 0; blue, 0 }  ][line width=0.75]    (10.58,-1.97) .. controls (8.19,-0.84) and (6.01,-0.18) .. (4.03,0) .. controls (6.01,0.18) and (8.19,0.84) .. (10.58,1.97)(6.56,-1.97) .. controls (4.17,-0.84) and (1.99,-0.18) .. (0,0) .. controls (1.99,0.18) and (4.17,0.84) .. (6.56,1.97)   ;
\draw  [draw opacity=0] (124.62,75.55) .. controls (124.62,75.69) and (124.63,75.83) .. (124.63,75.97) .. controls (124.63,86.31) and (116.26,94.71) .. (105.92,94.75) -- (105.85,75.97) -- cycle ; \draw  [color={rgb, 255:red, 0; green, 0; blue, 0 }  ,draw opacity=1 ] (124.62,75.55) .. controls (124.62,75.69) and (124.63,75.83) .. (124.63,75.97) .. controls (124.63,86.31) and (116.26,94.71) .. (105.92,94.75) ;  
\draw  [draw opacity=0] (195.88,94.75) .. controls (195.74,94.75) and (195.6,94.74) .. (195.46,94.74) .. controls (185.12,94.45) and (176.96,85.85) .. (177.22,75.51) -- (195.99,75.97) -- cycle ; \draw   (195.88,94.75) .. controls (195.74,94.75) and (195.6,94.74) .. (195.46,94.74) .. controls (185.12,94.45) and (176.96,85.85) .. (177.22,75.51) ;  
\draw  [draw opacity=0] (177.23,166.9) .. controls (177.22,166.76) and (177.22,166.62) .. (177.21,166.48) .. controls (177.01,156.14) and (185.21,147.58) .. (195.54,147.34) -- (195.99,166.11) -- cycle ; \draw   (177.23,166.9) .. controls (177.22,166.76) and (177.22,166.62) .. (177.21,166.48) .. controls (177.01,156.14) and (185.21,147.58) .. (195.54,147.34) ;  
\draw  [draw opacity=0] (105.03,147.35) .. controls (105.17,147.34) and (105.3,147.34) .. (105.44,147.34) .. controls (115.79,147.11) and (124.36,155.3) .. (124.62,165.64) -- (105.85,166.11) -- cycle ; \draw   (105.03,147.35) .. controls (105.17,147.34) and (105.3,147.34) .. (105.44,147.34) .. controls (115.79,147.11) and (124.36,155.3) .. (124.62,165.64) ;  
\draw  [color={rgb, 255:red, 208; green, 2; blue, 27 }  ,draw opacity=1 ][fill={rgb, 255:red, 208; green, 2; blue, 27 }  ,fill opacity=1 ] (117.6,87.58) .. controls (117.6,86.07) and (118.83,84.84) .. (120.34,84.84) .. controls (121.86,84.84) and (123.09,86.07) .. (123.09,87.58) .. controls (123.09,89.09) and (121.86,90.32) .. (120.34,90.32) .. controls (118.83,90.32) and (117.6,89.09) .. (117.6,87.58) -- cycle ;
\draw  [color={rgb, 255:red, 208; green, 2; blue, 27 }  ,draw opacity=1 ][fill={rgb, 255:red, 208; green, 2; blue, 27 }  ,fill opacity=1 ] (308.66,151.74) .. controls (308.66,150.23) and (309.89,149) .. (311.4,149) .. controls (312.92,149) and (314.15,150.23) .. (314.15,151.74) .. controls (314.15,153.26) and (312.92,154.48) .. (311.4,154.48) .. controls (309.89,154.48) and (308.66,153.26) .. (308.66,151.74) -- cycle ;
\draw [color={rgb, 255:red, 74; green, 144; blue, 226 }  ,draw opacity=1 ][line width=0.75]    (120.34,87.58) .. controls (152.73,136.15) and (150.65,141.85) .. (150.92,166.11) ;
\draw [color={rgb, 255:red, 74; green, 144; blue, 226 }  ,draw opacity=1 ]   (120.34,87.58) .. controls (135.92,99.08) and (150.83,97.65) .. (150.92,75.97) ;
\draw [color={rgb, 255:red, 74; green, 144; blue, 226 }  ,draw opacity=1 ]   (105.85,121.04) .. controls (131.81,121.89) and (130.86,107.16) .. (120.34,87.58) ;
\draw [color={rgb, 255:red, 74; green, 144; blue, 226 }  ,draw opacity=1 ]   (120.34,87.58) .. controls (157.95,123.79) and (175.06,122.84) .. (195.99,121.04) ;
\draw    (234.64,75.97) -- (234.64,166.11) ;
\draw [shift={(234.64,114.44)}, rotate = 90] [color={rgb, 255:red, 0; green, 0; blue, 0 }  ][line width=0.75]    (10.58,-1.97) .. controls (8.19,-0.84) and (6.01,-0.18) .. (4.03,0) .. controls (6.01,0.18) and (8.19,0.84) .. (10.58,1.97)(6.56,-1.97) .. controls (4.17,-0.84) and (1.99,-0.18) .. (0,0) .. controls (1.99,0.18) and (4.17,0.84) .. (6.56,1.97)   ;
\draw [color={rgb, 255:red, 0; green, 0; blue, 0 }  ,draw opacity=1 ][line width=0.75]    (234.64,75.97) -- (324.79,75.97) ;
\draw [shift={(283.32,75.97)}, rotate = 180] [color={rgb, 255:red, 0; green, 0; blue, 0 }  ,draw opacity=1 ][line width=0.75]    (6.56,-1.97) .. controls (4.17,-0.84) and (1.99,-0.18) .. (0,0) .. controls (1.99,0.18) and (4.17,0.84) .. (6.56,1.97)   ;
\draw    (234.64,166.11) -- (324.79,166.11) ;
\draw [shift={(283.32,166.11)}, rotate = 180] [color={rgb, 255:red, 0; green, 0; blue, 0 }  ][line width=0.75]    (6.56,-1.97) .. controls (4.17,-0.84) and (1.99,-0.18) .. (0,0) .. controls (1.99,0.18) and (4.17,0.84) .. (6.56,1.97)   ;
\draw    (324.79,75.97) -- (324.79,166.11) ;
\draw [shift={(324.79,114.44)}, rotate = 90] [color={rgb, 255:red, 0; green, 0; blue, 0 }  ][line width=0.75]    (10.58,-1.97) .. controls (8.19,-0.84) and (6.01,-0.18) .. (4.03,0) .. controls (6.01,0.18) and (8.19,0.84) .. (10.58,1.97)(6.56,-1.97) .. controls (4.17,-0.84) and (1.99,-0.18) .. (0,0) .. controls (1.99,0.18) and (4.17,0.84) .. (6.56,1.97)   ;
\draw  [draw opacity=0] (253.42,75.55) .. controls (253.42,75.69) and (253.42,75.83) .. (253.42,75.97) .. controls (253.42,86.31) and (245.06,94.71) .. (234.72,94.75) -- (234.64,75.97) -- cycle ; \draw  [color={rgb, 255:red, 0; green, 0; blue, 0 }  ,draw opacity=1 ] (253.42,75.55) .. controls (253.42,75.69) and (253.42,75.83) .. (253.42,75.97) .. controls (253.42,86.31) and (245.06,94.71) .. (234.72,94.75) ;  
\draw  [draw opacity=0] (324.67,94.75) .. controls (324.54,94.75) and (324.4,94.74) .. (324.26,94.74) .. controls (313.92,94.45) and (305.76,85.85) .. (306.01,75.51) -- (324.79,75.97) -- cycle ; \draw   (324.67,94.75) .. controls (324.54,94.75) and (324.4,94.74) .. (324.26,94.74) .. controls (313.92,94.45) and (305.76,85.85) .. (306.01,75.51) ;  
\draw  [draw opacity=0] (233.83,147.35) .. controls (233.97,147.34) and (234.1,147.34) .. (234.24,147.34) .. controls (244.59,147.11) and (253.16,155.3) .. (253.42,165.64) -- (234.64,166.11) -- cycle ; \draw   (233.83,147.35) .. controls (233.97,147.34) and (234.1,147.34) .. (234.24,147.34) .. controls (244.59,147.11) and (253.16,155.3) .. (253.42,165.64) ;  
\draw  [color={rgb, 255:red, 208; green, 2; blue, 27 }  ,draw opacity=1 ][fill={rgb, 255:red, 208; green, 2; blue, 27 }  ,fill opacity=1 ] (246.4,87.58) .. controls (246.4,86.07) and (247.63,84.84) .. (249.14,84.84) .. controls (250.66,84.84) and (251.89,86.07) .. (251.89,87.58) .. controls (251.89,89.09) and (250.66,90.32) .. (249.14,90.32) .. controls (247.63,90.32) and (246.4,89.09) .. (246.4,87.58) -- cycle ;
\draw [color={rgb, 255:red, 74; green, 144; blue, 226 }  ,draw opacity=1 ][line width=0.75]    (279.72,166.11) .. controls (280.1,142.89) and (290.08,140.51) .. (311.4,151.74) ;
\draw [color={rgb, 255:red, 74; green, 144; blue, 226 }  ,draw opacity=1 ]   (249.14,87.58) .. controls (274.87,107.24) and (279.62,93.94) .. (279.72,75.97) ;
\draw [color={rgb, 255:red, 74; green, 144; blue, 226 }  ,draw opacity=1 ]   (234.64,121.04) .. controls (263.47,120.55) and (272.02,121.5) .. (311.4,151.74) ;
\draw [color={rgb, 255:red, 74; green, 144; blue, 226 }  ,draw opacity=1 ]   (249.14,87.58) .. controls (286.75,123.79) and (303.39,121.5) .. (324.79,121.04) ;
\draw [color={rgb, 255:red, 74; green, 144; blue, 226 }  ,draw opacity=1 ][line width=0.75]    (249.14,87.58) .. controls (272.97,114.37) and (297.21,136.71) .. (311.4,151.74) ;
\draw (184.37,202.78) node [anchor=north west][inner sep=0.75pt]   [align=left] {\textsuperscript{}};
\end{tikzpicture}
    \centering
    \caption{Left: The graded surface $(S,\{p\},\eta)$ associated with $\k Q'/I'$ consists of a torus with one boundary and one marked point (in red) on that boundary, and the arc system $A'$ (in blue) consists of two arcs. Right: The graded surface $(S,M,\eta)$ associated with $\k Q/I$ consists of a torus with one boundary and two marked points (in red) on that boundary, and the arc system $A$ (in blue) consists of three arcs.}
    \label{figure:marked surfaces}
\end{figure}

We comment on the choice of grading $\eta$, c.f.~\cite[Section 6.6]{Alex22}. Pick two simple closed curves, $\alpha$ and $\beta$, on the torus $S$ which intersect in a single point (and thus give generators of $\pi_1(S)$).
Up to homotopy, $\eta$ is determined by its winding numbers $m\in\mathbb Z$ along $\alpha$ and $n\in\mathbb Z$ along $\beta$. Up to the action of the mapping class group $SL(2,\mathbb Z)$ of $S$, the gradings are thus classified by the non-negative integer $N=\mathrm{gcd}(m,n)$.
In particular, the constant line field $\eta$ is classified by $N=0$ and we call this the \textit{standard grading}.
In this case, the gradings on the three arcs in our full formal arc system $A$ can be chosen so that all arrows $a_1,a_2,b_1,b_2$ have degree zero, so the entire algebra is concentrated in degree zero.

Let $P$ be an arc on $S$ which connects the two points in $M$ and stays close to the boundary $\partial S$. There are two such arcs, up to isotopy, and for concreteness we take the one where the boundary is just to the left as we go from $p$ to $q$.
We fix an arbitrary grading on $P$ and consider it as an object of the Fukaya category $\cW(S,M,\eta)$. As such, it is an exceptional object.
An indecomposable object in $\cW(S,M,\eta)$ is supported on some immersed curve $c$ in $S$ which is either closed or has endpoints in $M$. Then $c$ is in $P^\perp$, i.e. $\mathrm{Hom}(P,c[n])=0$, if and only if $c$ avoids $q\in M$.
It follows that 
\begin{equation}
    P^\perp= \cW(S,\{p\},\eta)
\end{equation}
is the Fukaya category of a genus one surface with one boundary component and \textit{one} marked point on the boundary.

We want to find a graded gentle algebra corresponding to $P^\perp= \cW(S,\{p\},\eta)$.
For this choose the full formal arc system $A'$ on $(S,\{p\})$ with two arcs as indicated on the left in Figure~\ref{figure:marked surfaces}.
The graded gentle algebra corresponding $A'$ is described by the following  quiver $Q'$ with relations:
\[
Q'=\left\{
\begin{tikzcd}
\mathop{\bullet}\limits_{1} \arrow[r, "a", shift left=3.5] \arrow[r, "c", shift right=5] & \mathop{\bullet}\limits_{2} \arrow[l, "b"', shift left]
\end{tikzcd}\Bigg{\vert}\quad ab=bc=0
\right\}.
\]
In regards to the grading, it is important to note that the standard grading $N=0$ does not correspond to the case $|a|=|b|=|c|=0$, but rather, for example, to $|a|=|c|=0,|b|=1$. In general, we have $N=\mathrm{gcd}(1-|a|-|b|,1-|b|-|c|)$.

\section{Proofs of the main results}
\label{sec:main results}

In this section, we demonstrate that the category $P^{\perp}=\cW(S,\{p\},\eta)$ equipped with the standard grading $\eta$ does not admit a stability condition. Furthermore, we confirm that the category $\cW(S,M,\eta)=\mathrm{Perf}(\k Q)$ with non-standard grading does not satisfy the Jordan--H\"older property for polarizable SODs.

\subsection{The non-existence of stability conditions}

\begin{definition}[{\cite{Collins2009GluingSC}}]
A stability condition $\sigma=(Z,\cP)$ on $\cD$ is called \emph{reasonable} if it satisfies 
\[
\inf\limits_{E \text{ is semistable with } E\ne0}\vert Z(E)\vert>0.
\]
\end{definition}

We note that any stability condition satisfying the support property is reasonable.

\begin{lemma}[{\cite[Theorem 3.6]{Collins2009GluingSC}}]
\label{gluing-stab}
Let $\langle\cD_1,\cD_2\rangle$ be a semiorthogonal decomposition of a triangulated category of $\cD$. Let $\lambda_1$ denote the left adjoint to the inclusion $\cD_1 \hookrightarrow \cD$, and let $\rho_2$ denote the right adjoint to the inclusion $\cD_2 \hookrightarrow \cD$. Suppose that $(\sigma_1=(Z_1,\cP_1),\sigma_2=(Z_2,\cP_2))$ is a pair of reasonable stability conditions on $\cD_1$ and $\cD_2$, respectively. Assume that $\sigma_1$ and $\sigma_2$ fulfill the following conditions:
\begin{enumerate}
\item $\Hom^{\le0}_{\cD}(\cP_1(0,1],\cP_2(0,1])=0$;
\item There exists a real number $a\in(0,1)$ such that 
\[
\Hom^{\le0}_{\cD}(\cP_1(a,a+1],\cP_2(a,a+1])=0.
\]
\end{enumerate}
Then, there exists a unique \emph{reasonable} stability condition $\sigma=(Z,\cA)$ on $\cD$ glued from $\sigma_1$ and $\sigma_2$, where 
\[
\cA=\left\{E\in\cD\mid\lambda_1(E)\in\cP_1((0,1]), \rho_2(E)\in\cP_2((0,1])\right\},
\]
and 
\[
Z(E)=Z_1(\lambda_1(E))+Z_2(\rho_2(E)).
\]
\end{lemma}

\begin{definition}
The bounded derived category $P^{\perp}=D^b(Q')$ is said to have the \emph{standard grading} if the grading of the algebra $\k Q'/I'$ in \eqref{Q'} adheres to the following conditions:
\[
\vert\alpha\vert+\vert\beta\vert=1,\quad \vert\beta\vert+\vert\gamma\vert=1.
\]
\end{definition}

Recall that a stability condition $\sigma$ is said to have a \emph{gap} if the set $S^1 \setminus \Phi_{\sigma}$ includes an open interval, where $\Phi_{\sigma}$ denotes the collection of phases of $\sigma$-stable objects. Additionally, a stability condition $\sigma$ is termed \emph{finite-heart} if the heart of $\sigma$ has finitely many isomorphism classes of simple objects, and every object within the heart is of finite length.

\begin{lemma}[{\cite[Theorem A]{broomhead2024simpletiltslengthhearts}}]
\label{lem:gap=finite-heart}
Let $\cD$ be a triangulated category with $\rk K_0(\cD)<\infty$. Then a stability condition $\sigma\in\Stab(\cD)$ possesses a gap if and only if $\sigma\cdot z$ is finite-heart for some $z\in\mathbb{C}$.
\end{lemma}

\begin{theorem}
\label{thm:Existense of stab}
The bounded derived category $P^{\perp}=\cW(S,\{p\},\eta)$, where $\eta$ is the standard grading, does not admit a stability condition.
\end{theorem}
\begin{proof}
Let $\sigma_1=(Z_1,\cP_1)$ be a stability condition in $\Stab(P^{\perp})$. According to \cite[Theorem D]{chang2023recollementspartiallywrappedfukaya}, $P^{\perp}$ does not admit silting objects. Thus, by \cite[Theorem 6.1]{Koenig2012SiltingOS}, $\sigma_1$ is not finite-heart. From \Cref{lem:gap=finite-heart}, it follows that $\sigma$ does not have a gap. Consider the semiorthogonal decomposition
\[
D^b(Q)=\langle P^{\perp},P\rangle.
\]
Since $D^b(Q)$ is smooth and proper, we can apply \cite[Proposition 3.2]{halpernleistner2024stabilityconditionssemiorthogonaldecompositions}, which implies that there exists an integer $m \in \mathbb{Z}$ for which $\Hom^{\le m}(\cA_1, P) = 0$. Consequently, by \Cref{gluing-stab}, there exists a stability condition $\sigma_2=(Z_2,\cP_2)\in\Stab(P)$ such that $\sigma_1$ and $\sigma_2$ can be glued to form a reasonable stability condition $\sigma'=(Z,\cP)$ on $D^b(Q)$, which also does not have a gap. We consider the spaces of stability conditions on $P^{\perp}$ and $P$ with respect to the pairs $(\Lambda_1,v_1)$ and $(\Lambda_2,v_2)$, respectively. Let $\Vert\cdot\Vert_i$ denote the norm on $\Lambda_i$ that satisfies \eqref{support_property} for some $C_i>0$. Define $\Lambda:=\Lambda_1\oplus\Lambda_2$ and $v:=v_1\oplus v_2\colon K_0(D^b(Q))\to\Lambda$. Then, for any $\sigma$-semistable object $0\ne E\in\cP(\phi)$,
we have 
\begin{equation*}
\begin{aligned}
 \Vert v(E)\Vert &=\Vert v_1(\lambda_1(E))\Vert_1+\Vert v_2(\rho_2(E))\Vert_2\\
 &\le \vert C_1Z_1(\lambda_1(E))\vert+\vert C_2Z_2(\rho_2(E))\vert\\
 &\le C\left(\vert Z_1(\lambda_1(E))\vert+\vert Z_2(\rho_2(E)\vert\right)\\
 &= C\vert Z(E)\vert,
\end{aligned}
\end{equation*}
where $C=\max\left\{C_1,C_2\right\}$. Therefore, $\sigma\in\Stab(D^b(Q))$. Furthermore, due to the openness of the gap condition, there exists an open neighborhood $U\subset\Stab(D^b(Q))$ containing $\sigma$ such that every stability condition in $U$ does not have a gap.  By \cite[Lemma 67]{Alex22}, we conclude that any generic stability condition $\sigma\in\Stab(D^b(Q))$ possesses a gap, which contradicts the discussion above. 
\end{proof}

\begin{remark}
    The existence of a gap, \cite[Lemma 67]{Alex22}, relies on \cite[Lemma 66]{Alex22} whose proof is incomplete. The author of~\cite{Alex22} has communicated to us the following replacement of it, which turns out to also simplify the proof:

    Assume, for contradiction, that all the stable intervals have ends at the same marked boundary, then:
\begin{enumerate}
    \item Take an interval $X$ with ends in different marked boundaries, and nontrivial winding $(p,q)\neq (0,0)$.
    \item Decompose $X$ into its chain of stable intervals (COSI) decomposition. Among the components there has to be some interval $J$ with nontrivial winding, thus by assumption it has ends on the same marked boundary.
    \item Pick a basis so that $J$ has winding $(1,0)$ (or equivalently a way of cutting the torus into a rectangle with a puncture in the middle and the interval wrapping horizontally).
    \item Continue as in the original proof, i.e. pick $J'$ with winding $(0,1)$ and ends on different marked boundaries, decompose it, and find the forbidden polygon.
\end{enumerate}
\end{remark}

\begin{remark}
The challenge of proving the existence of stability conditions on the bounded derived categories of coherent sheaves for any smooth projective variety remains an open and difficult problem in the study of Bridgeland stability conditions. Verifying the existence of these stability conditions requires a detailed understanding of all hearts of bounded t-structures within a triangulated category $\cD$. Consequently, evidence of non-existence has been found only in rather straightforward cases. For example, a (quasi)-phantom category $\cD$ is defined by the property $K_0(\cD)\otimes\mathbb Q=0$, which  precludes the possibility of stability conditions, since the central charge function cannot map any object to $0\in\mathbb C$. As far as we know, \Cref{thm:Polarizable J-H} provides the first example of a smooth and proper triangulated category with a positive rank Grothendieck group that does not support any stability conditions.
\end{remark}

\subsection{The failure of the polarizable Jordan-H\"{o}lder property}

\begin{theorem}
\label{thm:Polarizable J-H} 
Let $S$ be a genus one surface with one boundary component, $M$ a pair of marked points on the boundary of $S$.
Then the category $\cW(S,M,\eta)\cong\mathrm{Perf}(\k Q)$, equipped with a non-standard grading $\eta$, does not fulfill the polarizable Jordan--H\"{o}lder property.
\end{theorem}

\begin{proof}
Consider the following semiorthogonal decompositions of $\cW(S,M,\eta)\cong\mathrm{Perf}(\k Q)$:
\[
\cW(S,M,\eta)=\langle P_1,P_2,P_{3}\rangle=\langle P^{\perp},P\rangle,
\]
where each $P_i$ denotes the projective module corresponding to a respective vertex. It is clear that the first semiorthogonal decomposition is maximal and polarizable. We claim that the second decomposition is also maximal and polarizable, which leads to the desired conclusion. According to {\cite[Theorem 1.1]{JSZ23}}, $P^{\perp}=\cW(S,\{p\},\eta)$ admits a silting object, which is equivalent to the existence of a bounded t-structure with \emph{length heart}, by \cite[Theorem 6.1]{Koenig2012SiltingOS}. 
According to \cite[Lemma 5.2]{Bridgeland2006SpacesOS}, there exists a nonempty subset $U\subset\Stab(P^{\perp})$ consisting of stability conditions with heart $\k Q'$. The fact that $\cW(S,\{p\},\eta)$ does not support nontrivial semiorthogonal decompositions is a consequence of the one-to-one correspondence between semiorthogonal decompositions of $P^{\perp}$ and proper sequences of good cuts of the associated marked surface $(S,\{p\},\eta)$, established in \cite[Theorem 3.20]{kopiva2022semiorthogonal}.
\end{proof}

\begin{remark} 
According to \Cref{gluing-stab}, we can build stability conditions on $\cW(S,M,\eta)$ through gluing, utilizing two distinct semiorthogonal decompositions of $\cW(S,M,\eta)$. It can be demonstrated that these resulting stability conditions reside within the same connected component of $\Stab(\cW(S,M,\eta))$ by describing the mutations of the S-graphs, as defined in \cite{HKK17}, from one semiorthogonal decomposition to the other. \Cref{figure:mutation} illustrates such a sequence of mutations of the S-graphs under a particular choice of non-standard grading $\eta$.
\begin{figure}
\tikzset{every picture/.style={line width=0.75pt}} 
\begin{tikzpicture}[x=0.75pt,y=0.75pt,yscale=-1,xscale=1]

\draw [line width=0.75]    (9.57,19.95) -- (108.36,18.86) ;
\draw [line width=0.75]    (18.26,19.3) -- (30.42,5.58) ;
\draw [line width=0.75]    (33.46,19.3) -- (45.62,5.58) ;
\draw [line width=0.75]    (76.9,19.3) -- (89.06,5.58) ;
\draw [line width=0.75]    (63,19.28) -- (75.16,5.55) ;
\draw [line width=0.75]    (47.58,18.91) -- (59.74,5.18) ;
\draw [line width=0.75]    (90.48,18.91) -- (102.64,5.18) ;
\draw [color={rgb, 255:red, 74; green, 144; blue, 226 }  ,draw opacity=1 ][line width=0.75]    (33.46,19.3) .. controls (20.1,86.92) and (84.15,83.26) .. (76.9,19.3) ;
\draw [color={rgb, 255:red, 0; green, 0; blue, 0 }  ,draw opacity=1 ][line width=0.75]    (33.46,19.3) .. controls (-22.03,85.78) and (23.59,99.81) .. (76.9,19.3) ;
\draw [line width=0.75]    (33.46,19.3) .. controls (78.58,88.98) and (144.8,96.69) .. (76.9,19.3) ;
\draw  [color={rgb, 255:red, 208; green, 2; blue, 27 }  ,draw opacity=1 ][fill={rgb, 255:red, 208; green, 2; blue, 27 }  ,fill opacity=1 ][line width=0.75]  (34.65,16.95) .. controls (35.63,17.82) and (35.9,19.57) .. (35.24,20.87) .. controls (34.59,22.17) and (33.26,22.53) .. (32.28,21.66) .. controls (31.3,20.79) and (31.03,19.04) .. (31.69,17.73) .. controls (32.34,16.43) and (33.67,16.08) .. (34.65,16.95) -- cycle ;
\draw  [color={rgb, 255:red, 208; green, 2; blue, 27 }  ,draw opacity=1 ][fill={rgb, 255:red, 208; green, 2; blue, 27 }  ,fill opacity=1 ][line width=0.75]  (78.09,16.95) .. controls (79.07,17.82) and (79.34,19.57) .. (78.68,20.87) .. controls (78.03,22.17) and (76.7,22.53) .. (75.72,21.66) .. controls (74.74,20.79) and (74.47,19.04) .. (75.13,17.73) .. controls (75.78,16.43) and (77.11,16.08) .. (78.09,16.95) -- cycle ;
\draw    (86.8,31.59) -- (78.7,37.46) ;
\draw [shift={(77.08,38.63)}, rotate = 324.08] [color={rgb, 255:red, 0; green, 0; blue, 0 }  ][line width=0.75]    (4.37,-1.32) .. controls (2.78,-0.56) and (1.32,-0.12) .. (0,0) .. controls (1.32,0.12) and (2.78,0.56) .. (4.37,1.32)   ;
\draw    (76.52,37.47) -- (67.92,35.64) ;
\draw [shift={(65.96,35.22)}, rotate = 12.03] [color={rgb, 255:red, 0; green, 0; blue, 0 }  ][line width=0.75]    (4.37,-1.32) .. controls (2.78,-0.56) and (1.32,-0.12) .. (0,0) .. controls (1.32,0.12) and (2.78,0.56) .. (4.37,1.32)   ;
\draw    (41.52,30.63) -- (33.27,37.57) ;
\draw [shift={(31.74,38.86)}, rotate = 319.93] [color={rgb, 255:red, 0; green, 0; blue, 0 }  ][line width=0.75]    (4.37,-1.32) .. controls (2.78,-0.56) and (1.32,-0.12) .. (0,0) .. controls (1.32,0.12) and (2.78,0.56) .. (4.37,1.32)   ;
\draw    (31.74,38.86) -- (21.84,38.24) ;
\draw [shift={(19.85,38.12)}, rotate = 3.56] [color={rgb, 255:red, 0; green, 0; blue, 0 }  ][line width=0.75]    (4.37,-1.32) .. controls (2.78,-0.56) and (1.32,-0.12) .. (0,0) .. controls (1.32,0.12) and (2.78,0.56) .. (4.37,1.32)   ;
\draw    (107.07,30.66) -- (130.26,30.66) ;
\draw [shift={(132.26,30.66)}, rotate = 180] [color={rgb, 255:red, 0; green, 0; blue, 0 }  ][line width=0.75]    (8.74,-2.63) .. controls (5.56,-1.12) and (2.65,-0.24) .. (0,0) .. controls (2.65,0.24) and (5.56,1.12) .. (8.74,2.63)   ;
\draw [line width=0.75]    (134.73,19.45) -- (233.51,18.36) ;
\draw [line width=0.75]    (143.41,18.8) -- (155.57,5.08) ;
\draw [line width=0.75]    (158.62,18.8) -- (170.78,5.08) ;
\draw [line width=0.75]    (202.06,18.8) -- (214.22,5.08) ;
\draw [line width=0.75]    (188.15,18.78) -- (200.31,5.05) ;
\draw [line width=0.75]    (172.74,18.41) -- (184.89,4.68) ;
\draw [line width=0.75]    (215.63,18.41) -- (227.79,4.68) ;
\draw [color={rgb, 255:red, 0; green, 0; blue, 0 }  ,draw opacity=1 ][line width=0.75]    (158.62,18.8) .. controls (145.25,86.42) and (209.3,82.76) .. (202.06,18.8) ;
\draw [color={rgb, 255:red, 74; green, 144; blue, 226 }  ,draw opacity=1 ][line width=0.75]    (158.62,18.8) .. controls (103.13,85.28) and (148.74,99.31) .. (202.06,18.8) ;
\draw [line width=0.75]    (158.62,18.8) .. controls (191.27,71.16) and (124.27,71.66) .. (158.62,18.8) ;
\draw  [color={rgb, 255:red, 208; green, 2; blue, 27 }  ,draw opacity=1 ][fill={rgb, 255:red, 208; green, 2; blue, 27 }  ,fill opacity=1 ][line width=0.75]  (159.8,16.45) .. controls (160.79,17.32) and (161.05,19.07) .. (160.4,20.37) .. controls (159.74,21.67) and (158.41,22.03) .. (157.43,21.16) .. controls (156.45,20.29) and (156.18,18.54) .. (156.84,17.23) .. controls (157.49,15.93) and (158.82,15.58) .. (159.8,16.45) -- cycle ;
\draw  [color={rgb, 255:red, 208; green, 2; blue, 27 }  ,draw opacity=1 ][fill={rgb, 255:red, 208; green, 2; blue, 27 }  ,fill opacity=1 ][line width=0.75]  (203.24,16.45) .. controls (204.23,17.32) and (204.49,19.07) .. (203.84,20.37) .. controls (203.18,21.67) and (201.86,22.03) .. (200.87,21.16) .. controls (199.89,20.29) and (199.62,18.54) .. (200.28,17.23) .. controls (200.93,15.93) and (202.26,15.58) .. (203.24,16.45) -- cycle ;
\draw    (156.74,37.26) -- (150.23,40.7) ;
\draw [shift={(148.46,41.64)}, rotate = 332.12] [color={rgb, 255:red, 0; green, 0; blue, 0 }  ][line width=0.75]    (4.37,-1.32) .. controls (2.78,-0.56) and (1.32,-0.12) .. (0,0) .. controls (1.32,0.12) and (2.78,0.56) .. (4.37,1.32)   ;
\draw    (201.67,36.97) -- (193.07,35.14) ;
\draw [shift={(191.11,34.72)}, rotate = 12.03] [color={rgb, 255:red, 0; green, 0; blue, 0 }  ][line width=0.75]    (4.37,-1.32) .. controls (2.78,-0.56) and (1.32,-0.12) .. (0,0) .. controls (1.32,0.12) and (2.78,0.56) .. (4.37,1.32)   ;
\draw    (164.91,31.65) -- (158.39,36.13) ;
\draw [shift={(156.74,37.26)}, rotate = 325.54] [color={rgb, 255:red, 0; green, 0; blue, 0 }  ][line width=0.75]    (4.37,-1.32) .. controls (2.78,-0.56) and (1.32,-0.12) .. (0,0) .. controls (1.32,0.12) and (2.78,0.56) .. (4.37,1.32)   ;
\draw    (148.46,41.64) -- (142.62,43.14) ;
\draw [shift={(140.69,43.64)}, rotate = 345.58] [color={rgb, 255:red, 0; green, 0; blue, 0 }  ][line width=0.75]    (4.37,-1.32) .. controls (2.78,-0.56) and (1.32,-0.12) .. (0,0) .. controls (1.32,0.12) and (2.78,0.56) .. (4.37,1.32)   ;
\draw    (231.41,30.66) -- (254.59,30.66) ;
\draw [shift={(256.59,30.66)}, rotate = 180] [color={rgb, 255:red, 0; green, 0; blue, 0 }  ][line width=0.75]    (8.74,-2.63) .. controls (5.56,-1.12) and (2.65,-0.24) .. (0,0) .. controls (2.65,0.24) and (5.56,1.12) .. (8.74,2.63)   ;
\draw [line width=0.75]    (259.63,20.02) -- (358.41,18.93) ;
\draw [line width=0.75]    (268.32,19.38) -- (280.48,5.65) ;
\draw [line width=0.75]    (283.52,19.38) -- (295.68,5.65) ;
\draw [line width=0.75]    (326.96,19.38) -- (339.12,5.65) ;
\draw [line width=0.75]    (313.06,19.35) -- (325.22,5.62) ;
\draw [line width=0.75]    (297.64,18.98) -- (309.8,5.25) ;
\draw [line width=0.75]    (340.54,18.98) -- (352.7,5.25) ;
\draw [color={rgb, 255:red, 0; green, 0; blue, 0 }, draw opacity=1 ][line width=0.75]    (283.52,19.38) .. controls (270.16,86.99) and (334.2,83.33) .. (326.96,19.38) ;
\draw [color={rgb, 255:red, 0; green, 0; blue, 0 }  ,draw opacity=1 ][line width=0.75]    (283.52,19.38) .. controls (228.03,85.85) and (273.65,99.88) .. (326.96,19.38) ;
\draw [color={rgb, 255:red, 74; green, 144; blue, 226 }  ,draw opacity=1 ][line width=0.75]    (283.52,19.38) .. controls (295.21,50.59) and (302.54,33.26) .. (326.96,19.38) ;
\draw  [color={rgb, 255:red, 208; green, 2; blue, 27 }  ,draw opacity=1 ][fill={rgb, 255:red, 208; green, 2; blue, 27 }  ,fill opacity=1 ][line width=0.75]  (284.71,17.02) .. controls (285.69,17.89) and (285.96,19.64) .. (285.3,20.94) .. controls (284.65,22.25) and (283.32,22.6) .. (282.34,21.73) .. controls (281.35,20.86) and (281.09,19.11) .. (281.74,17.81) .. controls (282.4,16.51) and (283.73,16.15) .. (284.71,17.02) -- cycle ;
\draw  [color={rgb, 255:red, 208; green, 2; blue, 27 }  ,draw opacity=1 ][fill={rgb, 255:red, 208; green, 2; blue, 27 }  ,fill opacity=1 ][line width=0.75]  (328.15,17.02) .. controls (329.13,17.89) and (329.4,19.64) .. (328.74,20.94) .. controls (328.09,22.25) and (326.76,22.6) .. (325.78,21.73) .. controls (324.79,20.86) and (324.53,19.11) .. (325.18,17.81) .. controls (325.84,16.51) and (327.17,16.15) .. (328.15,17.02) -- cycle ;
\draw    (316.02,35.29) -- (312.01,31.29) ;
\draw [shift={(310.59,29.88)}, rotate = 44.94] [color={rgb, 255:red, 0; green, 0; blue, 0 }  ][line width=0.75]    (4.37,-1.32) .. controls (2.78,-0.56) and (1.32,-0.12) .. (0,0) .. controls (1.32,0.12) and (2.78,0.56) .. (4.37,1.32)   ;
\draw    (326.57,37.54) -- (317.97,35.71) ;
\draw [shift={(316.02,35.29)}, rotate = 12.03] [color={rgb, 255:red, 0; green, 0; blue, 0 }  ][line width=0.75]    (4.37,-1.32) .. controls (2.78,-0.56) and (1.32,-0.12) .. (0,0) .. controls (1.32,0.12) and (2.78,0.56) .. (4.37,1.32)   ;
\draw    (289.81,32.22) -- (283.3,36.7) ;
\draw [shift={(281.65,37.83)}, rotate = 325.54] [color={rgb, 255:red, 0; green, 0; blue, 0 }  ][line width=0.75]    (4.37,-1.32) .. controls (2.78,-0.56) and (1.32,-0.12) .. (0,0) .. controls (1.32,0.12) and (2.78,0.56) .. (4.37,1.32)   ;
\draw    (281.65,37.83) -- (271.93,37.82) ;
\draw [shift={(269.93,37.82)}, rotate = 0.05] [color={rgb, 255:red, 0; green, 0; blue, 0 }  ][line width=0.75]    (4.37,-1.32) .. controls (2.78,-0.56) and (1.32,-0.12) .. (0,0) .. controls (1.32,0.12) and (2.78,0.56) .. (4.37,1.32)   ;
\draw    (33.56,110.16) -- (56.75,110.16) ;
\draw [shift={(58.75,110.16)}, rotate = 180] [color={rgb, 255:red, 0; green, 0; blue, 0 }  ][line width=0.75]    (8.74,-2.63) .. controls (5.56,-1.12) and (2.65,-0.24) .. (0,0) .. controls (2.65,0.24) and (5.56,1.12) .. (8.74,2.63)   ;
\draw [line width=0.75]    (60.56,99.7) -- (159.34,98.61) ;
\draw [line width=0.75]    (69.25,99.05) -- (81.41,85.33) ;
\draw [line width=0.75]    (84.45,99.05) -- (96.61,85.33) ;
\draw [line width=0.75]    (127.89,99.05) -- (140.05,85.33) ;
\draw [line width=0.75]    (113.99,99.03) -- (126.15,85.3) ;
\draw [line width=0.75]    (98.57,98.66) -- (110.73,84.93) ;
\draw [line width=0.75]    (141.47,98.66) -- (153.62,84.93) ;
\draw [color={rgb, 255:red, 0; green, 0; blue, 0 }  ,draw opacity=1 ][line width=0.75]    (84.45,99.05) .. controls (71.09,166.67) and (135.13,163.01) .. (127.89,99.05) ;
\draw [color={rgb, 255:red, 0; green, 0; blue, 0 }  ,draw opacity=1 ][line width=0.75]    (84.45,99.05) .. controls (31.07,163.76) and (120.27,164.16) .. (84.45,99.05) ;
\draw [color={rgb, 255:red, 74; green, 144; blue, 226 }  ,draw opacity=1 ][line width=0.75]    (84.45,99.05) .. controls (106.76,124.03) and (111.05,109.46) .. (127.89,99.05) ;
\draw  [color={rgb, 255:red, 208; green, 2; blue, 27 }  ,draw opacity=1 ][fill={rgb, 255:red, 208; green, 2; blue, 27 }  ,fill opacity=1 ][line width=0.75]  (85.64,96.7) .. controls (86.62,97.57) and (86.88,99.32) .. (86.23,100.62) .. controls (85.58,101.92) and (84.25,102.28) .. (83.27,101.41) .. controls (82.28,100.54) and (82.02,98.79) .. (82.67,97.48) .. controls (83.33,96.18) and (84.65,95.83) .. (85.64,96.7) -- cycle ;
\draw  [color={rgb, 255:red, 208; green, 2; blue, 27 }  ,draw opacity=1 ][fill={rgb, 255:red, 208; green, 2; blue, 27 }  ,fill opacity=1 ][line width=0.75]  (129.08,96.7) .. controls (130.06,97.57) and (130.32,99.32) .. (129.67,100.62) .. controls (129.02,101.92) and (127.69,102.28) .. (126.71,101.41) .. controls (125.72,100.54) and (125.46,98.79) .. (126.11,97.48) .. controls (126.77,96.18) and (128.09,95.83) .. (129.08,96.7) -- cycle ;
\draw    (96.05,109.56) -- (92.9,114.65) ;
\draw [shift={(91.85,116.36)}, rotate = 301.7] [color={rgb, 255:red, 0; green, 0; blue, 0 }  ][line width=0.75]    (4.37,-1.32) .. controls (2.78,-0.56) and (1.32,-0.12) .. (0,0) .. controls (1.32,0.12) and (2.78,0.56) .. (4.37,1.32)   ;
\draw    (128.85,108.36) -- (120.11,106.57) ;
\draw [shift={(118.15,106.17)}, rotate = 11.52] [color={rgb, 255:red, 0; green, 0; blue, 0 }  ][line width=0.75]    (4.37,-1.32) .. controls (2.78,-0.56) and (1.32,-0.12) .. (0,0) .. controls (1.32,0.12) and (2.78,0.56) .. (4.37,1.32)   ;
\draw    (91.85,116.36) -- (84.9,119.2) ;
\draw [shift={(83.05,119.96)}, rotate = 337.75] [color={rgb, 255:red, 0; green, 0; blue, 0 }  ][line width=0.75]    (4.37,-1.32) .. controls (2.78,-0.56) and (1.32,-0.12) .. (0,0) .. controls (1.32,0.12) and (2.78,0.56) .. (4.37,1.32)   ;
\draw    (83.05,119.96) -- (73.24,118.96) ;
\draw [shift={(71.25,118.76)}, rotate = 5.81] [color={rgb, 255:red, 0; green, 0; blue, 0 }  ][line width=0.75]    (4.37,-1.32) .. controls (2.78,-0.56) and (1.32,-0.12) .. (0,0) .. controls (1.32,0.12) and (2.78,0.56) .. (4.37,1.32)   ;
\draw    (158.06,110.41) -- (181.25,110.41) ;
\draw [shift={(183.25,110.41)}, rotate = 180] [color={rgb, 255:red, 0; green, 0; blue, 0 }  ][line width=0.75]    (8.74,-2.63) .. controls (5.56,-1.12) and (2.65,-0.24) .. (0,0) .. controls (2.65,0.24) and (5.56,1.12) .. (8.74,2.63)   ;
\draw [line width=0.75]    (183.99,99.41) -- (282.77,98.33) ;
\draw [line width=0.75]    (192.68,98.77) -- (204.83,85.04) ;
\draw [line width=0.75]    (207.88,98.77) -- (220.04,85.04) ;
\draw [line width=0.75]    (251.32,98.77) -- (263.48,85.04) ;
\draw [line width=0.75]    (237.42,98.74) -- (249.57,85.02) ;
\draw [line width=0.75]    (222,98.37) -- (234.16,84.65) ;
\draw [line width=0.75]    (264.9,98.37) -- (277.05,84.65) ;
\draw [color={rgb, 255:red, 0; green, 0; blue, 0 }  ,draw opacity=1 ][line width=0.75]    (207.88,98.77) .. controls (169.26,150.13) and (252.5,161.85) .. (207.88,98.77) ;
\draw [color={rgb, 255:red, 0; green, 0; blue, 0 }  ,draw opacity=1 ][line width=0.75]    (207.88,98.77) .. controls (132.4,139.56) and (219.26,176.13) .. (207.88,98.77) ;
\draw [line width=0.75]    (207.88,98.77) .. controls (270.97,153.85) and (278.12,118.42) .. (251.32,98.77) ;
\draw  [color={rgb, 255:red, 208; green, 2; blue, 27 }  ,draw opacity=1 ][fill={rgb, 255:red, 208; green, 2; blue, 27 }  ,fill opacity=1 ][line width=0.75]  (209.07,96.41) .. controls (210.05,97.28) and (210.31,99.04) .. (209.66,100.34) .. controls (209,101.64) and (207.68,101.99) .. (206.69,101.12) .. controls (205.71,100.26) and (205.45,98.5) .. (206.1,97.2) .. controls (206.76,95.9) and (208.08,95.55) .. (209.07,96.41) -- cycle ;
\draw  [color={rgb, 255:red, 208; green, 2; blue, 27 }  ,draw opacity=1 ][fill={rgb, 255:red, 208; green, 2; blue, 27 }  ,fill opacity=1 ][line width=0.75]  (252.51,96.41) .. controls (253.49,97.28) and (253.75,99.04) .. (253.1,100.34) .. controls (252.44,101.64) and (251.12,101.99) .. (250.13,101.12) .. controls (249.15,100.26) and (248.89,98.5) .. (249.54,97.2) .. controls (250.2,95.9) and (251.52,95.55) .. (252.51,96.41) -- cycle ;
\draw    (216.5,111.85) -- (210.59,114.69) ;
\draw [shift={(208.79,115.56)}, rotate = 334.29] [color={rgb, 255:red, 0; green, 0; blue, 0 }  ][line width=0.75]    (4.37,-1.32) .. controls (2.78,-0.56) and (1.32,-0.12) .. (0,0) .. controls (1.32,0.12) and (2.78,0.56) .. (4.37,1.32)   ;
\draw    (225.93,113.27) -- (219.94,116.83) ;
\draw [shift={(218.22,117.85)}, rotate = 329.35] [color={rgb, 255:red, 0; green, 0; blue, 0 }  ][line width=0.75]    (4.37,-1.32) .. controls (2.78,-0.56) and (1.32,-0.12) .. (0,0) .. controls (1.32,0.12) and (2.78,0.56) .. (4.37,1.32)   ;
\draw    (208.79,115.56) -- (200.2,116.72) ;
\draw [shift={(198.22,116.99)}, rotate = 352.3] [color={rgb, 255:red, 0; green, 0; blue, 0 }  ][line width=0.75]    (4.37,-1.32) .. controls (2.78,-0.56) and (1.32,-0.12) .. (0,0) .. controls (1.32,0.12) and (2.78,0.56) .. (4.37,1.32)   ;
\draw    (198.22,116.99) -- (186.5,116.5) ;
\draw [shift={(184.5,116.42)}, rotate = 2.39] [color={rgb, 255:red, 0; green, 0; blue, 0 }  ][line width=0.75]    (4.37,-1.32) .. controls (2.78,-0.56) and (1.32,-0.12) .. (0,0) .. controls (1.32,0.12) and (2.78,0.56) .. (4.37,1.32)   ;

\draw (83.21,39.01) node [anchor=north west][inner sep=0.75pt]  [font=\tiny]  {$1$};
\draw (66.2,40.49) node [anchor=north west][inner sep=0.75pt]  [font=\tiny]  {$1$};
\draw (37.48,38.89) node [anchor=north west][inner sep=0.75pt]  [font=\tiny]  {$2$};
\draw (28.07,41.52) node [anchor=north] [inner sep=0.75pt]  [font=\tiny]  {$1$};
\draw (151.79,42.26) node [anchor=north west][inner sep=0.75pt]  [font=\tiny]  {$1$};
\draw (191.58,39.1) node [anchor=north west][inner sep=0.75pt]  [font=\tiny]  {$2$};
\draw (160.23,37.52) node [anchor=north west][inner sep=0.75pt]  [font=\tiny]  {$1$};
\draw (141.7,45.85) node [anchor=north west][inner sep=0.75pt]  [font=\tiny]  {$1$};
\draw (273.5,40.43) node [anchor=north west][inner sep=0.75pt]  [font=\tiny]  {$1$};
\draw (316.49,39.67) node [anchor=north west][inner sep=0.75pt]  [font=\tiny]  {$1$};
\draw (285.94,36.49) node [anchor=north west][inner sep=0.75pt]  [font=\tiny]  {$1$};
\draw (307.35,34.17) node [anchor=north west][inner sep=0.75pt]  [font=\tiny]  {$1$};
\draw (75.83,122.51) node [anchor=north west][inner sep=0.75pt]  [font=\tiny]  {$1$};
\draw (120.15,109.57) node [anchor=north west][inner sep=0.75pt]  [font=\tiny]  {$1$};
\draw (86.86,120.37) node [anchor=north west][inner sep=0.75pt]  [font=\tiny]  {$1$};
\draw (96.08,114.25) node [anchor=north west][inner sep=0.75pt]  [font=\tiny]  {$1$};
\draw (188.97,119.37) node [anchor=north west][inner sep=0.75pt]  [font=\tiny]  {$1$};
\draw (223.58,118.14) node [anchor=north west][inner sep=0.75pt]  [font=\tiny]  {$1$};
\draw (202.01,118.94) node [anchor=north west][inner sep=0.75pt]  [font=\tiny]  {$1$};
\draw (211.65,116.67) node [anchor=north west][inner sep=0.75pt]  [font=\tiny]  {$1$};
\end{tikzpicture}
    \centering
    \caption{A sequence of mutations of the S-graphs, with each mutation being relative to a specific blue line.}
    \label{figure:mutation}
\end{figure}
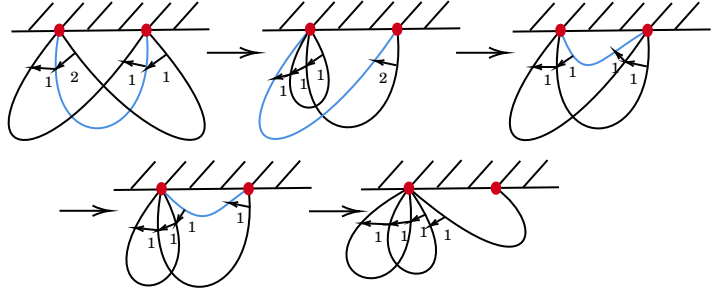
\end{remark}

\begin{remark}
In general, we may consider a graded gentle algebra $A$ associated with the corresponding graded marked surface $(S,M,\eta)$, where $S$ is a surface of genus $g\ge1$ with one boundary component $\partial S\cong S^1$ and two marked points $M=\{p,q\}$ on the boundary. Denote by $P$ the exceptional object in $\cW(S,M,\eta)\cong\mathrm{Perf}(A)$ determined by the arc as illustrated in \Cref{figure:general_P}. Similarly, there exists a graded gentle algebra $A'$ associate with the graded surface $(S,\{p\},\eta)$ such that 
 \[
 P^{\perp}\cong\cW(S,\{p\},\eta)\cong\mathrm{Perf}(A').
 \]
 Furthermore, we have the following semiorthogonal decompositions of $\cW(S,M,\eta)$:
\[
\cW(S,M,\eta)=\langle P_1,\dots,P_{n}\rangle=\langle P^{\perp},P\rangle,
\]
where each $P_i$ represents the projective module corresponding to a respective vertex. Following the argument in \Cref{thm:Polarizable J-H}, we can further conclude that the category $\cW(S,M,\eta)$ does not satisfy the polarizable Jordan--H\"{o}lder property when $g\ge2$.

 \begin{figure}
\centering
\tikzset{every picture/.style={line width=0.75pt}}
\begin{tikzpicture}[x=0.75pt,y=0.75pt,yscale=-1,xscale=1]
\draw   (217.79,92.78) .. controls (219.75,83.96) and (232.66,76.8) .. (246.63,76.8) .. controls (260.61,76.8) and (270.35,83.96) .. (268.4,92.78) .. controls (266.45,101.61) and (253.53,108.77) .. (239.56,108.77) .. controls (225.59,108.77) and (215.84,101.61) .. (217.79,92.78) -- cycle ;
\draw    (236.57,78.03) .. controls (251.02,72.55) and (420.38,40.51) .. (417.57,106.05) .. controls (414.76,171.58) and (228.91,111.71) .. (221.29,103.02) ;
\draw  [draw opacity=0] (301.97,84.52) .. controls (301.96,84.47) and (301.96,84.41) .. (301.96,84.36) .. controls (301.99,82.13) and (306.83,80.37) .. (312.78,80.43) .. controls (318.4,80.48) and (322.99,82.14) .. (323.47,84.2) -- (312.74,84.46) -- cycle ; \draw   (301.97,84.52) .. controls (301.96,84.47) and (301.96,84.41) .. (301.96,84.36) .. controls (301.99,82.13) and (306.83,80.37) .. (312.78,80.43) .. controls (318.4,80.48) and (322.99,82.14) .. (323.47,84.2) ;  
\draw  [draw opacity=0] (326.56,80.71) .. controls (324.59,85.11) and (319.13,88.24) .. (312.72,88.2) .. controls (306.31,88.15) and (300.9,84.93) .. (299.01,80.5) -- (312.81,77.08) -- cycle ; \draw   (326.56,80.71) .. controls (324.59,85.11) and (319.13,88.24) .. (312.72,88.2) .. controls (306.31,88.15) and (300.9,84.93) .. (299.01,80.5) ;  
\draw  [draw opacity=0] (332.68,112.38) .. controls (332.68,112.32) and (332.68,112.27) .. (332.68,112.21) .. controls (332.7,109.99) and (337.54,108.23) .. (343.49,108.29) .. controls (349.11,108.34) and (353.71,110) .. (354.18,112.05) -- (343.45,112.32) -- cycle ; \draw   (332.68,112.38) .. controls (332.68,112.32) and (332.68,112.27) .. (332.68,112.21) .. controls (332.7,109.99) and (337.54,108.23) .. (343.49,108.29) .. controls (349.11,108.34) and (353.71,110) .. (354.18,112.05) ;  
\draw  [draw opacity=0] (357.27,108.57) .. controls (355.31,112.96) and (349.84,116.1) .. (343.44,116.05) .. controls (337.03,116) and (331.61,112.78) .. (329.72,108.36) -- (343.52,104.94) -- cycle ; \draw   (357.27,108.57) .. controls (355.31,112.96) and (349.84,116.1) .. (343.44,116.05) .. controls (337.03,116) and (331.61,112.78) .. (329.72,108.36) ;  
\draw  [color={rgb, 255:red, 208; green, 2; blue, 27 }  ,draw opacity=1 ][fill={rgb, 255:red, 208; green, 2; blue, 27 }  ,fill opacity=1 ] (266.97,91.36) .. controls (266.97,90.57) and (267.61,89.93) .. (268.4,89.93) .. controls (269.19,89.93) and (269.83,90.57) .. (269.83,91.36) .. controls (269.83,92.14) and (269.19,92.78) .. (268.4,92.78) .. controls (267.61,92.78) and (266.97,92.14) .. (266.97,91.36) -- cycle ;
\draw  [color={rgb, 255:red, 208; green, 2; blue, 27 }  ,draw opacity=1 ][fill={rgb, 255:red, 208; green, 2; blue, 27 }  ,fill opacity=1 ] (216.37,91.36) .. controls (216.37,90.57) and (217.01,89.93) .. (217.79,89.93) .. controls (218.58,89.93) and (219.22,90.57) .. (219.22,91.36) .. controls (219.22,92.14) and (218.58,92.78) .. (217.79,92.78) .. controls (217.01,92.78) and (216.37,92.14) .. (216.37,91.36) -- cycle ;
\draw  [color={rgb, 255:red, 0; green, 0; blue, 0 }  ,draw opacity=1 ][fill={rgb, 255:red, 0; green, 0; blue, 0 }  ,fill opacity=1 ] (384.35,99.97) .. controls (384.35,99.35) and (384.84,98.86) .. (385.46,98.86) .. controls (386.07,98.86) and (386.57,99.35) .. (386.57,99.97) .. controls (386.57,100.58) and (386.07,101.08) .. (385.46,101.08) .. controls (384.84,101.08) and (384.35,100.58) .. (384.35,99.97) -- cycle ;
\draw  [color={rgb, 255:red, 0; green, 0; blue, 0 }  ,draw opacity=1 ][fill={rgb, 255:red, 0; green, 0; blue, 0 }  ,fill opacity=1 ] (376.35,99.97) .. controls (376.35,99.35) and (376.84,98.86) .. (377.46,98.86) .. controls (378.07,98.86) and (378.57,99.35) .. (378.57,99.97) .. controls (378.57,100.58) and (378.07,101.08) .. (377.46,101.08) .. controls (376.84,101.08) and (376.35,100.58) .. (376.35,99.97) -- cycle ;
\draw  [color={rgb, 255:red, 0; green, 0; blue, 0 }  ,draw opacity=1 ][fill={rgb, 255:red, 0; green, 0; blue, 0 }  ,fill opacity=1 ] (368.57,100.19) .. controls (368.57,99.58) and (369.07,99.08) .. (369.68,99.08) .. controls (370.3,99.08) and (370.79,99.58) .. (370.79,100.19) .. controls (370.79,100.8) and (370.3,101.3) .. (369.68,101.3) .. controls (369.07,101.3) and (368.57,100.8) .. (368.57,100.19) -- cycle ;
\draw [color={rgb, 255:red, 74; green, 144; blue, 226 }  ,draw opacity=1 ]   (268.4,91.36) .. controls (287.95,76.31) and (281.09,68.6) .. (256.52,73.46) ;
\draw [color={rgb, 255:red, 74; green, 144; blue, 226 }  ,draw opacity=1 ] [dash pattern={on 3pt off 3pt on 3pt off 3pt}]  (256.52,73.46) .. controls (239.09,79.17) and (232.23,82.03) .. (217.79,91.36) ;
\end{tikzpicture}
    \caption{The graded surface $(S,M,\eta)$ consists of a surface of genus $g\ge1$ with one boundary and two marked points in red. The object $P\in\cW(S,M,\eta)$ is determined by the arc in blue.}
    \label{figure:general_P}
\end{figure}
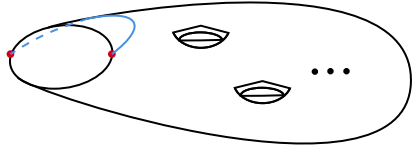
\end{remark}

\bibliographystyle{alpha}
\bibliography{sample}

\Addresses

\end{document}